  \documentclass[11pt,letterpaper,twoside]{article}

  \usepackage[margin=1.2in]{geometry}
  \usepackage{amsmath,amssymb,amsthm,amsfonts,xspace,xcolor}
  \usepackage{parskip,fancyhdr,url}
  \usepackage{fourier}

  \usepackage[
     breaklinks,colorlinks,
     citecolor=blue,linkcolor=blue,urlcolor=teal
  ]{hyperref}

  \usepackage{tikz}
  \usetikzlibrary{arrows,calc,decorations.pathmorphing,backgrounds,positioning,fit}

  \usepackage[all]{xy}
  \begingroup
    \makeatletter
    \@for\theoremstyle:=definition,remark,plain\do{%
      \expandafter\g@addto@macro\csname th@\theoremstyle\endcsname{%
        \addtolength\thm@preskip\parskip
        }%
      }
  \endgroup

  \newcommand\address[1]{}
  \newcommand\email[1]{}
  \newcommand\dedicatory[1]{}

  \theoremstyle{plain}
  \newtheorem{theorem}{Theorem}[section]
  \newtheorem{proposition}[theorem]{Proposition}
  \newtheorem{corollary}[theorem]{Corollary}
  \newtheorem{lemma}[theorem]{Lemma}

  \newtheorem{introthm}{Theorem}
  \newtheorem{introcor}[introthm]{Corollary}

  \theoremstyle{definition}

  \newtheorem*{claim*}{Claim}

  \newtheorem*{question*}{Question}
  \newtheorem*{answer*}{Answer}
  \newtheorem*{application*}{Application}

  \newcommand{\secref}[1]{Section~\ref{Sec:#1}}
  \newcommand{\thmref}[1]{Theorem~\ref{Thm:#1}}
  \newcommand{\corref}[1]{Corollary~\ref{Cor:#1}}
  \newcommand{\lemref}[1]{Lemma~\ref{Lem:#1}}
  \newcommand{\propref}[1]{Proposition~\ref{Prop:#1}}

  \newcommand{\figref}[1]{Figure~\ref{Fig:#1}}

  \renewcommand{\eqref}[1]{Equation~(\ref{Eq:#1})}

  \newcommand{\lab}{\ensuremath{\ell}\xspace}
  
  \DeclareMathOperator{\init}{i}
  \DeclareMathOperator{\supp}{supp}
  \DeclareMathOperator{\term}{t}

  \newcommand{\calG}{\mathcal{G}}

  \newcommand{\calT}{\mathcal{T}}

  \newcommand{\from}{\colon\,}

  \newcommand{\G}{\ensuremath{\calG_n}\xspace} 
  \newcommand{\T}{\ensuremath{\calT_n}\xspace} 

  \newcommand{\bT}{\ensuremath{\overline{\Theta}}\xspace} 
  \newcommand{\NN}{\ensuremath{\mathbb{N}}\xspace} 
  \newcommand{\ZZ}{\ensuremath{\mathbb{Z}}\xspace} 
  \newcommand{\GL}{\ensuremath{\operatorname{GL}_n(\ZZ)}\xspace} 
  \newcommand{\F}{\ensuremath{\operatorname{F}_n}\xspace} 
  \newcommand{\FF}{\ensuremath{\operatorname{F}}\xspace} 
  \newcommand{\mcg}{\ensuremath{\operatorname{MCG}(\Sigma)}\xspace}  
  \newcommand{\aut}{\ensuremath{\operatorname{Aut}(\F)}\xspace} 
  \newcommand{\out}{\ensuremath{\operatorname{Out}(\F)}\xspace} 
  \newcommand{\norm}[1]{\ensuremath{\left\| {#1} \right\|}\xspace} 
  \newcommand{\gen}[1]{\ensuremath{\left\langle {#1} \right\rangle}\xspace} 
  \renewcommand{\split}{\ensuremath{\stackrel{s}{\to}}\xspace}
  \newcommand{\splitt}[1]{\ensuremath{\stackrel{#1}{\longrightarrow}}\xspace}

  \newcommand{\param}{{\mathchoice{\mkern1mu\mbox{\raise2.2pt\hbox{$
  \centerdot$}}
  \mkern1mu}{\mkern1mu\mbox{\raise2.2pt\hbox{$\centerdot$}}\mkern1mu}{
  \mkern1.5mu\centerdot\mkern1.5mu}{\mkern1.5mu\centerdot\mkern1.5mu}}}
  \fancypagestyle{plain}

\begin{document}

%section{Title and abstract}

  \title    {Uniform growth rate}
  \author   {Kasra Rafi}
  \address  {Department of Mathematics\\
             University of Toronto\\
             Toronto, CA} 
  \email    {rafi@math.toronto.edu}
  \author   {Jing Tao}
  \address  {Department of Mathematics\\
             University of Utah\\
             Norman, OK 73019-0315, USA}
  \email    {jing@math.ou.edu}
  \date     {}
  \author   {Kasra Rafi\footnote{\small Partially supported by NCERC
  Research Grant, RGPIN 435885.} $\,\,$ and Jing Tao\footnote{\small
  Partially supported by NSF Research Grant, DMS-1311834}}

  \maketitle
  \thispagestyle{empty}

  \begin{abstract} 
    
    \noindent 
    
    In an evolutionary system in which the rules of mutation are local in
    nature, the number of possible outcomes after $m$ mutations is an
    exponential function of $m$ but with a rate that depends only on the
    set of rules and not the size of the original object. We apply
    this principle to find a uniform upper bound for the growth rate of
    certain groups including the mapping class group. We also find a
    uniform upper bound for the growth rate of the number of homotopy
    classes of triangulations of an oriented surface that can be obtained
    from a given triangulation using $m$ diagonal flips. 
    
  \end{abstract}
  
  %\tableofcontents

\section{Introduction}

  Let $G$ be a group and $S$ be a generating set for $G$. We denote the
  word length in $G$ associated to $S$ with $\norm{\param}_S$. Recall that
  the growth rate of $G$ (relative to $S$) is defined to be 
  \[ 
    h_G = \lim_{R \to \infty} \frac{\log \, \# B_R(G)}{R},
    \qquad\text{where}\qquad
    B_R(G) = \Big\{g \in G  \ \Big| \ \norm{g}_S\leq R \Big\}.
  \] 

  In his $60^{\text th}$ birthday conference, Bill Thurston mentioned that
  the mapping class group has a growth rate that is independent of its
  genus. Namely, consider the following set of curves on a surface
  $\Sigma=\Sigma_{g,p}$ of genus $g$ with $p$ punctures: 
   
  \begin{figure}[htp]
  \begin{center}
    \includegraphics[width=.8\textwidth]{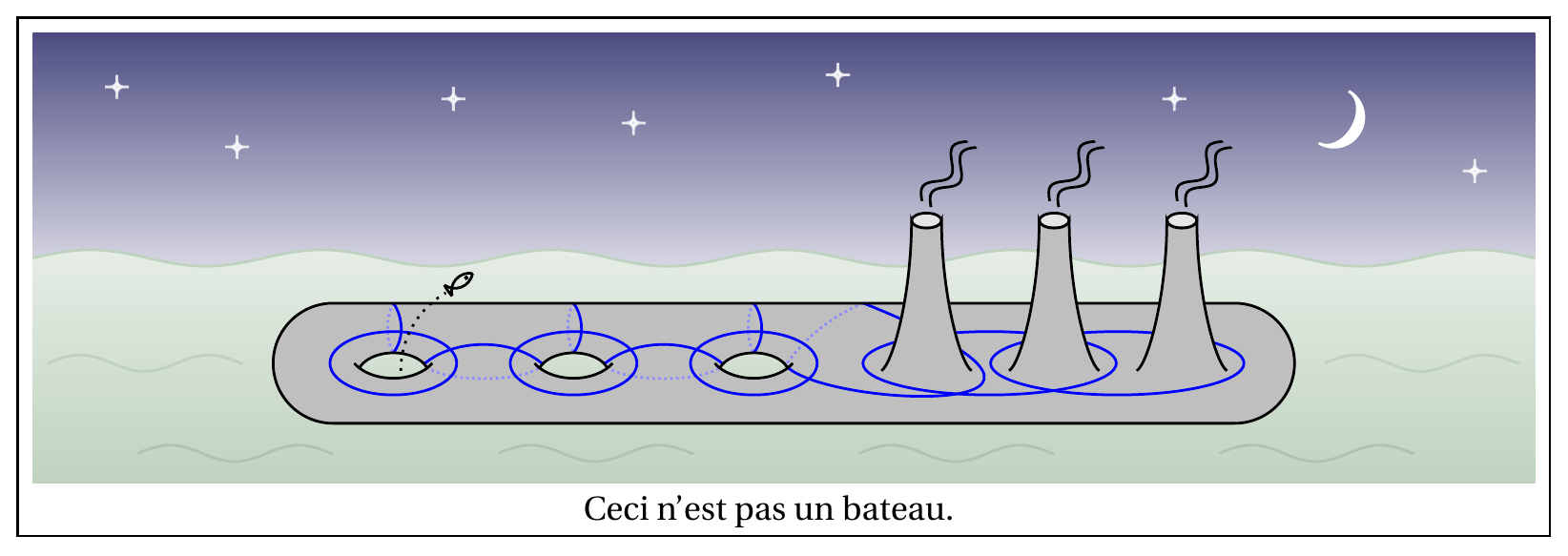}
  \end{center}
  %\caption{Ce n'est pas une surface.}
  \end{figure}
   
  Let $S$ be the set of Dehn (or half) twists around these curves. This set
  $S$ generates \mcg, the mapping class group of $\Sigma$ \cite{Lic64,
  FM12, Art47, Bir74}. (Note that $S$ is a combination of the Lickorish
  generators of the mapping class group of a closed surface and the
  standard generators of a braid group). We will refer to $S$ as the set of
  \emph{extended Lickorish generators} for \mcg. Then the growth rate of
  \mcg equipped with the word metric associated to $S$ has an upper bound
  that is independent of the topology of $\Sigma$. This, Thurston asserted,
  is true since most pairs of elements in $S$ commute.

  Note that, in fact, the number of elements in $S$ that do not commute
  with a given element in $S$ is uniformly bounded. We show that this is
  enough to obtain the uniform growth rate in general. 
   
  \begin{introthm} \label{Thm:IntroCommute} 
     
    Given any $c_0$, let $S$ be a generating set for a group $G$ such that,
    for every $s \in S$, the number elements of $S$ that do not commute
    with $s$ is bounded by $c_0$. Then $h_G \leq \log (2c_0+2) +1$. 

  \end{introthm}  
   
  Since each curve in the extended Lickorish generators intersects at most
  $3$ other curves, we obtain:
   
  \begin{introcor} \label{Cor:MCG}

    The growth rate of \mcg relative to the extended Lickorish generators
    is bounded by $\log8+1$. 

  \end{introcor}
   
  Uniform growth rate can also be shown regarding groups \aut, \out, \GL
  and similar groups if the generating set is chosen such that the number
  of generators that do not commute with a given generator is uniformly
  bounded. In fact, these groups have natural generating sets with this
  property. For example, in the case of \aut, let \F be the free group with
  basis $\{a_1,\ldots a_n\}$, and consider the following three types of
  automorphisms of \F. 

  \begin{enumerate}
    \item Inversion: For $1\le i \le n$, $I_i(a_i) = \overline{a_i}$ and
    fixes all other $a_j$. 
    \item Transposition: For $1 \le i \le n-1$, $P_i(a_i) = a_{i+1}$ and
    $P_i(a_{i+1}) = a_i$ and fixes all other $a_j$.
    \item Multiplication: For $1 \le i \le n-1$, $M_i(a_i) = a_ia_{i+1}$
    and fixes all other $a_j$. 
  \end{enumerate}

  The collection of inversions, transpositions, and multiplications
  generate \aut \cite{MKS66,LS77} and is called the set of \emph{local
  Nielsen generators}. For each $s \in S$, the number of elements that do
  not commute with $s$ is at most $7$, we obtain:

  \begin{introcor}
    
    The growth rate of \aut relative to local Nielsen generators
    is bounded by $\log 16+1$.

  \end{introcor}

  \subsection{Evolving structures}
   
  Another context to apply this philosophy is the setting of evolving
  structures. We follow the footsteps of the work of
  Sleator-Tarjan-Thurston \cite{STT92} where they showed that if a graph is
  allowed to evolve using a set of rule that change the graph locally, then
  the growth rate of the number of possible outcomes after $R$ mutations is
  bounded by a constant depending on the rules of evolution and not the
  size of the graph. This was used in \cite{STT92} to estimate the diameter
  of the space of plane triangulations equipped with the diagonal flip
  metric and in \cite{RT13} to estimate the diameter of the space of cubic
  graphs equipped with the Whitehead move metric. Similar to their work,
  one can also consider the evolution of labeled graphs. Generalizing the
  results in \cite{STT92} slightly, we prove:
  
  \begin{introthm} \label{Thm:IntroGraph}

    Let $G$ be any group and $\Gamma$ be a $G$--labeled trivalent graph
    (see \secref{Graph} for definition). Let $B_R(\Gamma)$ be the set of
    $G$--labeled graphs that are obtained from $\Gamma$ by at most $R$
    splits. Then, \[ \lim_{R \to \infty}\frac{\log \# B_R(\Gamma)}{R} \leq
    3 \log 4. \] That is, the growth rate of $B_R(\Gamma)$ is independent
    of the size and shape of the starting graph $\Gamma$ and of the group
    $G$. 

  \end{introthm}
  
  As an application, we can prove a combinatorial version of \corref{MCG}.
  Namely, let $\T(\Sigma)$ be the space of homotopy classes of
  triangulations of the surface $\Sigma$ with $n$ vertices.

  \begin{introthm} \label{Thm:IntroTriangle}
  
    For $T \in \T(\Sigma)$, let $B_R(T)$ be the set of triangulations in
    $\T(\Sigma)$ that are obtained from $T$ using $R$ diagonal flips. Then
    \[ \lim_{R \to \infty} \frac{\log \# B_R(T)}{R} \le 3 \log 4 \] for
    every surface $\Sigma$ and any number of vertices $n$. 

  \end{introthm}
   
  Note that, even though \thmref{IntroTriangle} is a direct analogue of
  \corref{MCG} it does not follow from it. This is because the quotient of
  $\T(\Sigma)$ by \mcg has a size that goes to infinity as the number of
  vertices $n$ approaches infinity. 

  \subsection{Remarks and references}

  Our \thmref{IntroCommute} follows immediately from an upper bound on the
  growth rate of a right-angled Artin group $A(\Theta)$ with defining graph
  $\Theta$, in terms of the maximum degree of the complementary graph \bT
  (\thmref{RAAG}). Other results relating the growth rate of $A(\Theta)$ to
  the shape of $\Theta$ have been obtained in the past. For instance, it
  was shown in \cite{Sco07} that the growth series of $A(\Theta)$ can be
  computed in terms of the clique polynomial of $\Theta$. Similar results
  can be found in \cite{AP14} and \cite{McM14}. However, the degree of \bT
  cannot be recovered from the coefficients of the clique polynomial of
  $\Theta$, so these results are independent from ours. 
  
  Our proof of \thmref{RAAG} is related to normal forms for elements of a
  right-angled Artin group. A normal form for a word representing an
  element in $A(\Theta)$ is obtained by shuffling commuting elements and
  removing inverse pairs of generators of $A(\Theta)$ whenever possible
  (\cite{HM95}). By fixing an ordering of $V(\Theta)$, then every element
  of $A(\Theta)$ admits a unique normal form, obtained by additionally
  shuffling lower-order letters to lower positions whenever possible. In
  our proof of \thmref{RAAG}, we construct a \emph{canonical
  representative} for a given word, obtained similarly by shuffling
  lower-order letters to lower positions. However, we do not need to cancel
  inverse pairs, so the canonical representative of a word may not be in
  normal form. 

  \subsection*{Acknowledgements} 

  We thank Benson Farb for useful conversations following the talk by
  Thurston. We thank the GEAR network for their support. We also thank the
  referee for helpful comments. 

\section{Uniform growth rates}
  
  \label{Sec:RAAGs}

  \subsection{Preliminaries}
  
  Let $G$ be a finitely generated group. By convention, the inverse of an
  element $g \in G$ will be represented by $\overline{g}$; and for any
  subset $S \subset G$, let $\overline{S} = \{ \overline{s} \colon s \in
  S\}$.  A \emph{word} in $S \cup \overline{S}$ is a sequence
  $w=[s_1,\ldots,s_R]$, where $s_i \in S \cup \overline{S}$; $R$ is the
  \emph{length} of $w$. We allow the empty word whose length is $0$. A word
  $w=[s_1,\ldots,s_R]$ represents an element $g \in G$ if $g=s_1 \cdots
  s_R$. (The empty word represents the identity element.) By a
  \emph{generating set} for $G$ we will mean a finite set $S \subset G
  \setminus \{1\}$ such that every element $g \in G$ is represented by a
  word in $S \cup \overline{S}$. The \emph{word length} $\norm{g}_S$ of $g$
  relative to a generating set $S$ is the length of the shortest word in $S
  \cup \overline{S}$ representing $g$. For any $R$, $B_R(G)$ is the set of
  elements of $G$ with word length at most $R$. The \emph{growth rate}
  (also called the entropy) of $G$ relative to $S$ is \[ h_G = \lim_{R \to
  \infty} \frac{\log \# B_R(G)}{R},\] where the above limit exists by
  sub-additivity.

  We remark that the growth rate of $G$ depends on the generating set, but
  positivity of the growth rate does not. The growth rate of $\FF_n$
  relative to a basis is $\log(2n-1)$. If $G$ contains a subgroup
  isomorphic to $\FF_2$, then $h_G$ is strictly positive. See \cite{GdlH97}
  and the references within for more details.

  \subsection{RAAGs} 

  A graph is a 1-dimensional CW complex. It is \emph{simple} if there are
  no self-loops or double edges. 
  
  Let $\Theta$ be a finite simple graph. Let $V(\Theta)$ and $E(\Theta)$ be
  the set of vertices and edges of $\Theta$. An element of $E(\Theta)$ will
  be denoted by $vw$, where $v$ and $w$ are the vertices of the edge. The
  \emph{complementary graph} of $\Theta$ is the graph $\bT$ with $V(\Theta)
  = V(\bT)$ but two vertices span an edge in $\bT$ if and only if they do
  not in $\Theta$.   
  
  The \emph{right-angled Artin group} or RAAG associated to $\Theta$ is the
  group $A(\Theta)$ with the presentation: 
  \[ 
    A(\Theta) = \gen{s_v \text{ for } v \in V(\Theta) \colon [s_v,s_w]=1
    \text{ for } vw \in E(\Theta) }.
  \] 
  The collection $S=\{s_v\}$ will be called the \emph{standard generating
  set} of $A(\Theta)$. We will often ignore the distinction between a
  vertex $v$ and the generator $s_v$.
  
  \begin{theorem} \label{Thm:RAAG}
   
   If the valence of every vertex in $\bT$ is bounded above by a constant
   $c_0$, then the growth rate of $A(\Theta)$ relative to the standard
   generating set $S$ is bounded by $\log (2c_0 + 2)+1$.
   
  \end{theorem}
  
  From \thmref{RAAG}, we derive \thmref{IntroCommute} as corollary.

  \begin{corollary} \label{Cor:Commute} 
     
    Let $S_G$ be a generating set for a group $G$ such that for every $s \in
    S_G$, the number elements of $S_G$ that do not commute with $s$ is
    bounded by $c_0$. Then $h_G \leq \log (2c_0+2)+1$. 

  \end{corollary}  
  
  \begin{proof}
     
    Let $\Theta$ be the graph with vertex set $S_G$ and $ss' \in E(\Theta)$
    if and only if $[s,s']=1$ in $G$. The natural map from $A(\Theta)$ to
    $G$ taking the standard generating set $S$ to $S_G$ extends to a
    surjective homomorphism, and the hypothesis on $S_G$ implies the
    valence of every vertex in $\bT$ is bounded by $c_0$. All together, we
    obtain $h_G \le h_{A(\Theta)} \le \log(2c_0+2)+1$. \qedhere
    
  \end{proof}

  The rest of the section is dedicated to proving \thmref{RAAG}.  
  
  Given a word $w=[s_1,\ldots,s_R]$ in $S \cup \overline{S}$, the $j$--th
  letter of $w$ is $w(j) = s_j$. A word $w'=[t_1,\dots,t_R]$ is a
  \emph{reordering} of $w$ if $t_1 \cdots t_R = s_1\cdots s_R$ and there is
  a permutation $\sigma$ such that $t_j = s_{\sigma(j)}$. We say the letter
  $s_k$ in $w$ is \emph{ready} for position $i$, $i \le k$,  if $s_k$
  commutes with every $s_j$, for $i \le j\le k$. 
   
  At every vertex $v$ of $\bT$, label the half-edges at $v$ from $1$ to
  $d_v$, where $d_v \le c_0$ is the valence of $v$. Let $n$ be the
  cardinality of $V(\bT)$. Fix a labeling $L_0 \from V(\bT) \to \NN$ whose
  image is $\{1,\ldots n\}$. 

  Fix $w_0=[s_1,\ldots,s_R]$. We will inductively construct a sequence
  $w_1,\ldots,w_R$ of words that reorders $w_0$, in conjunction with a
  sequence $L_1,\ldots,L_R$ of labeling of $V(\bT)$. The final word $w_R$
  will be called the \emph{canonical representative} of $g=s_1 \cdots s_R$
  induced by $w_0$. ($W_R$ depends on $W_0$). Along this process, we
  produce an encoding of the canonical representative by a sequence of
  integers $\ell_1,\ldots,\ell_R$. 
  
  Suppose for $i \ge 0$, $w_i=[u_1,\ldots,u_i,t_{i+1},\ldots, t_R]$, a
  labeling $L_i$ of $V(\bT)$, and a sequence $\ell_1, \ldots, \ell_i$ are
  given. Among $\{t_{i+1},\ldots,t_R\}$, let $U$ be the subset of letters
  that are ready for position $i+1$. Pick $t \in U$ such that $L_i(t)$ is
  minimal among all elements of $U$. Set \[ w_{i+1} = \left[
  u_1,\ldots,u_i,t,t_{i+1},\ldots, \hat{t}, \ldots,t_R \right], \] $u_{i+1}
  = t$. If $t \in S$, then let $\ell_{i+1}=L_i(t)$; if $t \in
  \overline{S}$, then let $\ell_{i+1} = -L_i(t)$. The word $w_{i+1}$ is a
  reordering of $w_i$ and hence of $w_0$ by induction.
 
  We now define the labeling $L_{i+1} \from V(\bT) \to \NN$. Let $n_i$ be
  the largest value of $L_i$. Let $(e_1,\ldots,e_d)$ be the half-edges of
  \bT incident at $t$ listed in order, where $d$ is the valence of $t$. For
  each $e_k$, let $v_k$ be the vertex connected to $t$ by the edge
  associated to $e_k$. We set $L_{i+1}(v_k) = n_i+k$ for each
  $k=1,\ldots,d$, and $L_{i+1}(v) = L_i(v)$ for all other $v \in V(\bT)$. 

  \begin{lemma} \label{Lem:Canonical}
    
    Let $n=\#V(\bT)$. Then 
    \[ 1 \le \left| \lab_1 \right| \le \left| \lab_2 \right| \le
    \cdots \le \left| \lab_R \right| \le n+c_0R. \] 
     
  \end{lemma}
  
  \begin{proof}
    
    For each $i \ge 1$, we show $|\lab_i| \le |\lab_{i+1}|$. Let
    $w_R=[u_1,\ldots,u_R]$. We have $|\lab_i|=L_i(u_i)$. For $v \in
    V(\bT)$, $L_i(v) = L_{i+1}(v)$ unless $v$ is in the link of $u_i$; in
    the latter case, $L_{i+1}(v)$ is is bigger than the maximal value of
    $L_i$ If $u_i$ and $u_{i+1}$ do not commute, then $u_{i+1}$ is in the
    link of $u_i$, therefore $L_{i+1}(u_{i+1})$ exceeds the maximal value
    of $L_i$, and in particular $L_{i+1}(u_{i+1}) > L_i(u_i)$. If $u_i$ and
    $u_{i+1}$ commute, then they were both ready for position $i$. In this
    case, $L_i(u_{i+1}) = L_{i+1}(u_{i+1})$, and $u_i$ was chosen precisely
    because its label $L_i(u_i)$ is minimal among all elements in the set
    $u_{i+1},\ldots,u_R$ that were ready for position $i$. We conclude
    $|\lab_i| \le |\lab_{i+1}|$.  
    
    The largest value of $L_{i+1}$ is at most $c_0$ plus the largest value
    of $L_i$. Hence the largest value of $L_R$ is at most $n+c_0R$. This
    bounds all $|\ell_i|$. \qedhere

  \end{proof}
  
  Set $C_R = n+c_0 R$. Let $D_R = \big\{ \pm1, \pm 2, \ldots, \pm C_R,
  C_R+1 \big\}$ and let 
  \[ 
    W_R = \big\{ \left( \lab_1,\ldots,\lab_R \right) \colon \lab_i \in D_R
    \text{ and } |\lab_1| \le \cdots \le |\lab_R| \big\}.
  \] 

  \begin{proposition}

    There exists an embedding of $B_R(G)$ into $W_R$, hence $\#B_R(G) \le
    \#W_R$.
  
  \end{proposition}
  
  \begin{proof}
  
    Let $g \in B_R(G)$ have $\norm{g}_S=r$. Pick any word
    $w=[s_1,\ldots,s_r]$ representing $g$ and let $w_r$ be the canonical
    representative of $g$ induced from $w$. Let $(\ell_1,\ldots,\ell_r)$ be
    the code of $w_r$. If $r<R$, then extend the sequence to $\left(
    \lab_1, \ldots, \lab_r, \lab_{r+1}, \ldots, \lab_R \right)$ by setting
    $\lab_{r+i} = C_R+1$ for all $i=1,\ldots R-r$. By \lemref{Canonical},
    $(\ell_1,\ldots,\ell_R) \in W_R$. This gives a map $B_R(G) \to W_R$. 
    
    To see this is an embedding, we show how to recover $w_r$ and hence $g$
    from the sequence $\left( \lab_1,\ldots,\lab_R \right)$. Recall \bT is
    equipped with a cyclic ordering of the half-edges at every vertex and a
    labeling $L_0$ of the vertices from $1$ to $n$. Let $w_0$ be the empty
    word. Suppose for $0 \le i \le r-1$, $L_i \from V(\bT) \to \NN$ and a
    word $w_i=[u_1,\ldots,u_i]$ are defined. If $\lab_{i+1}=C_R+1$, then
    set $u_{i+1} = u_{i+2} = \cdots u_R = 1 $. Otherwise, Let $v$ be the
    unique vertex in \bT with label $|\lab_{i+1}| =L_i(v)$. Set $u_{i+1}=v$
    if $\lab_{i+1}$ is positive and $u_{i+1} = \overline{v}$ if
    $\lab_{i+1}$ is negative. Let $(v_1,\ldots,v_d)$ be the vertices in the
    link of $u_{i+1}$ listed in cyclic order. Let $n_i$ be the largest
    value of $L_i$. Construct $L_{i+1} \from V(\bT) \to \NN$ by setting
    $L_{i+1}(v_k)=n_i+k$ and $L_{i+1}(u)=L_i(u)$ for all other $u \in \bT$.
    Then $w_r=[u_1,\ldots,u_r]$. \qedhere
  
  \end{proof}

  We now give an upper bound for the growth rate of $\#W_R$, which will
  complete the proof of \thmref{RAAG}.

  \begin{lemma}
   
  $\lim_{R \to \infty} \frac{\log \#W_R}{R} \le \log (2c_0+2)+1$. 
  
  \end{lemma}
  
  \begin{proof}
   
   Suppose $p(R)$ and $q(R)$ are two functions of $R$ with $\lim_{R \to
   \infty} \frac{p(R)}{q(R)} = \frac{1}{\epsilon}$. Then, using Stirling's
   formula, $\log \begin{pmatrix} p \\ q \end{pmatrix}$ is asymptotic to $p
   H(\epsilon)$ as $R \to \infty$, where
   \[ H(\epsilon) = \epsilon \log \frac{1}{\epsilon} + (1-\epsilon) \log
   \frac{1}{1-\epsilon} \]
   is the binary entropy function. (See \cite[Ch. 1]{Mac03}.)

   For any $R \ge 1$ and $C\ge R$,  by a simple counting argument, the set
   \[ W(R,C) = \big\{ (x_1,\ldots,x_R) \colon x_i \in
   \{1,\ldots,C\}, x_1 \le \cdots \le x_R \big\}\]  has cardinality
   $\#W(R,C) = \begin{pmatrix} C+R-1 \\ R \end{pmatrix}.$

   Let $C=C_R+1=n+c_0R+1$. We have: \[ \#W_R \le 2^R \begin{pmatrix}
   n+R(c_0+1)\\ R \end{pmatrix} \quad \text{and} \quad \lim_{R \to \infty}
   \frac{n+R(c_0+1)}{R} = c_0 + 1. \]
   Therefore,
   \begin{align*} 
     \lim_{R \to \infty} \frac{\log \#W_R}{R}
     & \le 
     \lim_{R \to \infty} \frac{\log 2^R
     \begin{pmatrix} 
     n+R(c_0+1) \\ R 
     \end{pmatrix} }{R} 
    = \lim_{R \to \infty} \frac{ R\log 2 + \Big(  n+R(c_0+1)
    \Big) H\left( \frac{1}{c_0+1} \right)}{R}\\
    & = \log 2 + (c_0+1) H\left( \frac{1}{c_0+1} \right) 
    = \log 2 + \log(c_0+1) + c_0 \log \left(1+\frac{1}{c_0} \right)
    \\ 
    &\le \log (2 c_0+2) + 1. \qedhere
   \end{align*} 

  \end{proof}
  
\section{Evolving structures on $G$--Labeled graphs}
  
  \label{Sec:Graph}

  A graph is is oriented if each edge is oriented. For any edge $e$ of an
  oriented graph, denote by $\init(e)$ and $\term(e)$ the initial and
  terminal vertex of $e$. If $e$ is a \emph{loop}, then $\init(e) =
  \term(e)$. The orientation of $e$ induces an orientation on each
  half-edge of $e$: the half-edge $e_l$ containing $\init(e)$ is oriented
  so that $\init(e_l) = \init(e)$ ($\term(e_l)$ is a point in the interior
  of $e$), and the half-edge $e_r$ containing $\term(e)$ is oriented so
  that $\term(e_r) = \term(e)$.   

  Given a group $G$, an oriented graph is $G$--labeled if each edge is
  labeled by an element of $G$. Two $G$--labeled graphs are
  \emph{equivalent} if one is obtained from the other by reversing the
  orientation of some of the edges and relabeling those edges with the
  inverse words. This defines an equivalent relation on the set of
  $G$--labeled graphs. Fix $n$ and let $\G(G)$ be the set of equivalent
  classes of trivalent $G$--labeled graphs of rank $n$. (Recall the rank of
  a graph is the rank of its fundamental group.) We now consider operations
  that derive from an element in $\G(G)$ another element in $\G(G)$. 
 
  \begin{figure}[htp!]
    \begin{center}
      \includegraphics{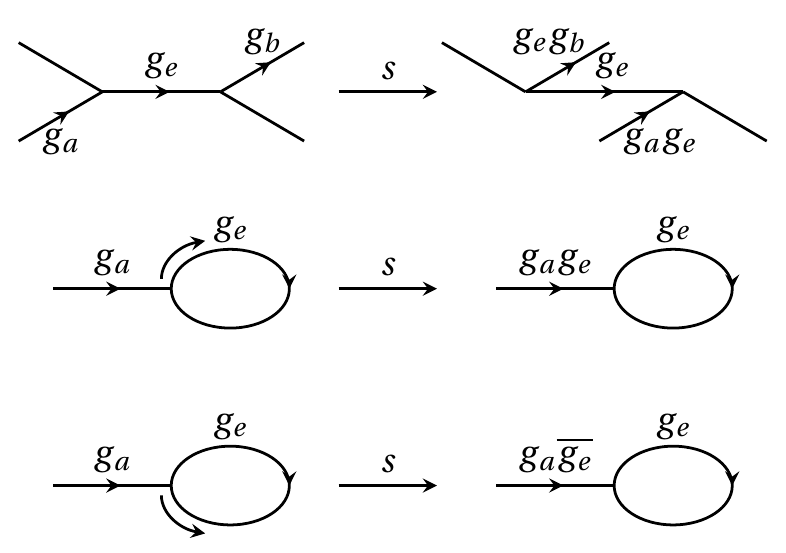}
    \end{center}
    \caption{Double split and loop split}
    \label{Fig:Split}
  \end{figure}
  
  Let $\Gamma \in \G(G)$. Let $e$ be an edge with label $g_e$. There are
  two \emph{types} of edges in $\Gamma$: loop or non-loop. First assume $e$
  is not a loop. Choose a half-edge $\tilde{a}$ (not a half-edge of $e$)
  incident at $i(e)$, and let $a$ be the edge associated to $\tilde{a}$
  with label $g_a$. Disconnect $\tilde{a}$ from $\init(e)$ and reattach it
  to $\term(e)$, while changing the label $g_a \to g_a g_e$ if
  $\term(\tilde{a}) = \init(e)$, or $g_a \to \overline{g_e} g_a$ if
  $\init(\tilde{a}) = \init(e)$. We call this a \emph{forward split} along
  $e$. A forward split along $e$ is well defined: if we reverse the
  orientation of $a$ and invert $g_a$, then the resulting graph is
  equivalent. Similarly, take a half-edge $\tilde{b}$ incident at
  $\term(e)$. A \emph{backward split} along $e$ is obtained by
  disconnecting $\tilde{b}$ from $\term(e)$ and reattaching it to
  $\init(e)$, while changing the label $g_b \to g_e g_b$ if $\init(\tilde
  b) = \term(e)$, or $g_b \to g_b \overline{g_e}$ if $\term(\tilde{b}) =
  \init(e)$. This is again well defined. A \emph{double split} along $e$
  (see \figref{Split}) is the composition of a forward and a backward split
  along $e$. The resulting graph from a double split is trivalent, unlike
  from a backward or forward split alone. If we reverse the orientation of
  $e$ and invert $g_e$, then a forward split along $e$ becomes a backward
  split along $e$ and vice versa. Therefore, a double split is well-defined
  on the equivalent class of $\Gamma$.
  
  For a loop $e$, let $a$ be the edge connected to $e$ with label $g_a$. A
  \emph{forward split} along $e$ changes the label $g_a \to g_a g_e$ if
  $\term(a) = \init(e)$, or $g_a \to \overline{g_e} g_a$ if $\init(a) =
  \init(e)$. A \emph{backward split} changes the label $g_a \to g_a
  \overline{g_e}$ if $\term(a) = \init(e)$, or $g_a \to g_e g_a$ if
  $\init(a) = \term(e)$ (see \figref{Split}). By a \emph{loop split} along
  $e$ we will mean either a forward or a backward split along $e$. This is
  again well defined on the equivalent class of $\Gamma$.
 
  For any edge $e$ of $\Gamma$, a \emph{split} along $e$ will mean either a
  double split or a loop split depending on the type of $e$. We will
  represent a split along $e$ by $\Gamma \split \Gamma'$ and call
  $e=\supp(s)$ the \emph{support} of $s$. 

  Fix $\Gamma_0 \in \G(G)$. A derivation $D=[s_1,\ldots,s_R]$ of length $R$
  is a sequence of splits \[ \Gamma_0 \splitt{s_1} \Gamma_1 \splitt{s_2}
  \cdots \splitt{s_R} \Gamma_R. \] Set $\Gamma_i =
  [s_1,\ldots,s_i](\Gamma_0)$; also write $\Gamma_R = D(\Gamma_0)$. We will
  say $\Gamma_R$ is derived from $\Gamma_0$ by $D$. Let $B_R(\Gamma_0)$ be
  the set of all trivalent graphs (up to equivalence) that are derived from
  $\Gamma_0$ by a derivation of length at most $R$. Our main result is
  \thmref{IntroGraph}, restated below. 
    
  \begin{theorem} \label{Thm:Graph}
    
    For any $\Gamma_0 \in \G(G)$, \[ \lim_{R \to \infty} \frac{\log \#
    B_R(\Gamma_0)}{R} \le 3 \log 4.\]
    
  \end{theorem}
  
  The main idea behind the proof of \thmref{Graph} is to give a normal form
  for a derivation. This was done in \cite{STT92}  for unlabeled graphs. It
  turns out labeled graphs do not pose significant additional difficulties.
  We are also careful to obtain an explicit upper bound for the growth rate
  of $\#B_R(\Gamma)$. 

  \begin{figure}[htp!]
    \begin{center}
      \includegraphics{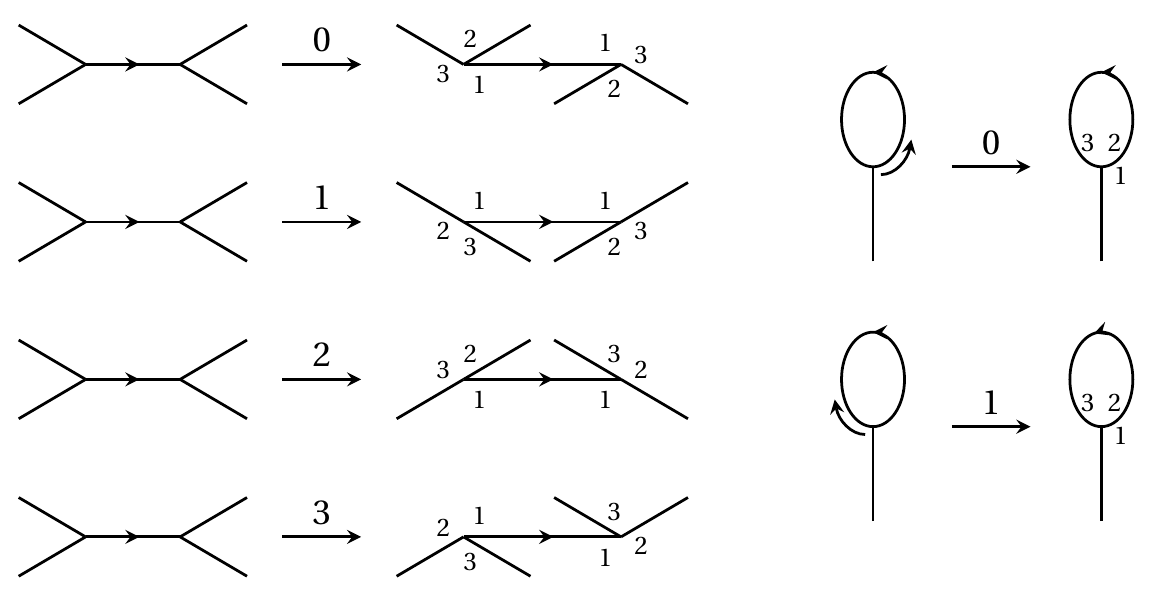}
    \end{center}
    \caption{Configurations of splits}
    \label{Fig:Configurations}
  \end{figure}
  
  A split $\Gamma \splitt{s} \Gamma'$ defines a bijection between the edges
  of $\Gamma$ and $\Gamma'$. For any edge $e$ in $\Gamma$, let $s(e)$ be
  its image in $\Gamma'$. We say $\supp(s)$ and its vertices are
  \emph{destroyed} by $s$, and $s(\supp(s))$ and its vertices are
  \emph{created} by $s$. All other vertices of $\Gamma$ \emph{survive} $s$.
  An edge of $\Gamma$ survives $s$ if all of its vertices survive $s$.

  Given $\Gamma \splitt{s} \Gamma'$. If a representative of $\Gamma$ is
  chosen, then $s$ naturally induces representative for $\Gamma'$. Fix a
  representative in the equivalent class of $\Gamma_0$. This way, for any
  derivation $D=[s_1,\ldots,s_R]$, we can inductively define a
  representative for each $\Gamma_i = [s_1,\ldots,s_i](\Gamma_0)$.

  We will refer to \figref{Configurations} for the following discussion.
  For each vertex $v$ of $\Gamma_0$, label the half-edges at $v$ from $1$
  to $3$ so they can be cyclically ordered. Let $D=[s_1,\ldots,s_R]$ be a
  derivation. We will cyclically order the half-edges at each vertex of
  $\Gamma_i=[s_1,\ldots,s_i](\Gamma_0)$ and \emph{label} each $s_i$ as
  follows. Let $X$ be a fixed planar binary tree with four valence-$1$
  vertices. The distinguished middle edge of $X$ is oriented (see
  \figref{Configurations}). Let $P$ be the planar graph which is the wedge
  of an interval and an oriented loop (also see \figref{Configurations}).
  Now suppose for $i\ge 0$, the half-edges of $\Gamma_i$ are labeled. Let
  $e$ be the support of $s_{i+1}$ in $\Gamma_i$. If $e$ is not a loop, then
  the cyclic ordering at $\init(e)$ and $\term(e)$ allows us to identify a
  contractible neighborhood of $e$ with $X$. The four configurations in the
  left column of \figref{Configurations} represent all possible double
  splits with support the middle edge of $X$. Record the label of the
  configuration that $s_{i+1}$ identifies with; this is the label of
  $s_{i+1}$. Similarly, if $e$ is a loop, then identify a neighborhood of
  $e$ with $P$. Label $s_{i+1}$ by 0 if $s_{i+1}$ is a forward split along
  $e$ and label $s_{i+1}$ by $1$ otherwise (see \figref{Configurations}).
  Let $\ell \in \{0,1,2,3\}$ be the label of $s_{i+1}$. Note that we always
  know if $e$ is a loop or not so there is no confusion with the
  duplication of the labels $0$ and $1$. For each vertex $v$ of
  $\Gamma_{i+1}$, if $v$ is not created by $s_{i+1}$, then the half-edges
  at $v$ will inherit their labels from $\Gamma_i$; otherwise, label the
  half-edges at $v$ from $1$ to $3$ according to the right side of
  configuration $\lab$ in \figref{Configurations}. 
  
  Let $D=[s_1,\ldots,s_R]$. Compute the label of each $s_i$ from above. Fix
  $i \ge 0$ and let $e$ be any edge in $\Gamma_i$. For $k > i+1$, we will
  say $e$ \emph{survives} $[s_{i+1},\ldots,s_{k-1}]$ if for all
  $j=i+1,\ldots,k-1$, the image of $e$ in $\Gamma_{j-1}$ survives $s_j$. In
  particular, $e$ remains the same type from $\Gamma_i$ to $\Gamma_{k-1}$.
  Let $e_i$ be the preimage of $\supp(s_k)$ in $\Gamma_i$. We say $s_k$ is
  \emph{ready} for $\Gamma_i$ if $e_i$ survives $[s_{i+1},\ldots,s_{k-1}]$.
  In this case, we can apply $s_k$ to $\Gamma_i$ with support $e_i$ using
  the label of $s_k$; this is well-defined since $e_i$ is the same type as
  $\supp(s_k)$.

  Consider \[ \Gamma_{k-2} \splitt{s_{k-1}} \Gamma_{k-1} \splitt{s_k}
  \Gamma_k.\] Suppose $s_k$ is ready for $\Gamma_{k-2}$. Apply $s_k$ to
  $\Gamma_{k-2}$ and let $\Gamma_{k-1}'$ be the resulting graph. Propagate
  the labels of half-edges from $\Gamma_{k-2}$ to $\Gamma_{k-1}'$ as
  before. Since $e_{k-2}$ and $e=\supp(s_{k-1})$ are disjoint in
  $\Gamma_{k-2}$, $e$ survives $s_k$, so we may apply $s_{k-1}$ to
  $\Gamma_{k-1}'$ with support $s_k(e)$ using the label of $s_{k-1}$. Let
  \[ \Gamma_{k-2} \splitt{s_k} \Gamma_{k-1}' \splitt{s_{k-1}} \Gamma_k'\]
  be the derivation obtained by switching the order of $s_{k-1}$ and $s_k$.
  We claim the following:
  
  \begin{lemma} \label{Lem:Commute}
    
    With the same notation as above. If $s_k$ is ready for $\Gamma_{k-2}$,
    then $s_{k-1}$ and $s_k$ commute; that is, $\Gamma_k = \Gamma_k'$.

  \end{lemma}

  \begin{figure}[htp!]
    \begin{center}
      \includegraphics{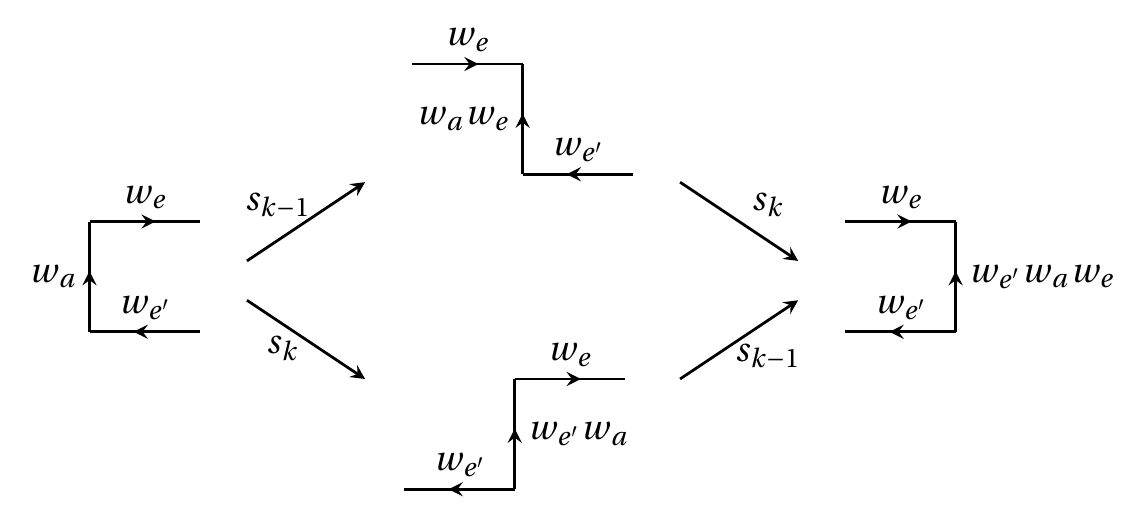}
    \end{center}
    \caption{If $s_k$ is ready for position $k-2$, then $s_{k-1}$ and $s_k$
    commute. }
    \label{Fig:Commute}
  \end{figure}

  \begin{proof}

    For any split $s$, we say an edge is \emph{affected} by $s$ if its
    label is changed by $s$. Any split affects at most two edges. Let $e$
    and $e'$ be the supports of $s_{k-1}$ and $s_k$ in $\Gamma_{k-1}$
    respectively. If $s_{k-1}$ and $s_k$ do not affect the same edge in
    $\Gamma_{k-2}$, then they clearly commute. So let $a$ be an edge in
    $\Gamma_{k-2}$ affected by both $s_{k-1}$ and $s_k$. $a$ must share a
    vertex with both $e$ and $e'$. Since $e$ and $e'$ are disjoint, $a$
    cannot be a loop. The proof that the labels of $s_{k-1} \circ s_k(a)$
    and $s_k \circ s_{k-1}(a)$ are the same now follows from considering
    different cases. The proof in all cases are similar. \figref{Commute}
    shows the case when neither $e$ and $e'$ are loops and $\term(e') =
    \init(a)$ and $\term(a) = \init(e)$. Since $s_k \circ s_{k-1}$ and
    $s_{k-1} \circ s_k$ affect the edges labels the same way, $\Gamma_k' =
    \Gamma_k$. \qedhere
    
  \end{proof}
   
  Let $D=[s_1,\ldots,s_R]$. For any $i \le k-1$, if $s_k$ is ready for
  $\Gamma_i$, then $s_k$ is ready for $\Gamma_j$, for all $i \le j \le
  k-1$. By applying \lemref{Commute} $k-i$ times, we see that 
  \[ D' = [s_1,\ldots,s_i,s_k,s_{i+1},\ldots, \widehat{s_k},
  \ldots, s_R] \] is a well defined derivation and $D(\Gamma_0) =
  D'(\Gamma_0)$.

  Set $D_0 = D$ and let $\Gamma_R = D(\Gamma_0)$. Inductively, we will
  construct a sequence $D_1$, \ldots, $D_R$ of derivations such that
  $D_j(\Gamma_0) = \Gamma_R$ for all $j=1,\ldots,R$. The final derivation
  $D_R$ will be called the \emph{canonical derivation} of $\Gamma_R$ coming
  from $D_0$. 
  
  Let $M=2n-2$ be the number of vertices of $\Gamma_0$. \emph{Label} each
  vertex of $\Gamma_0$ by a distinct integer from $1$ to $M$. Similarly,
  \emph{label} the edges of $\Gamma_0$ from $1$ to $N$, where $N=3n-3$ is
  the number of edges of $\Gamma_0$. 
 
  Suppose for $i \ge 0$, $D_i=[u_1,\ldots,u_i,t_{i+1},\ldots,t_R]$ has been
  constructed. Also, suppose the vertices and edges of $\Gamma_i'$ have
  been labeled, where $\Gamma_i' = [u_1,\ldots,u_i](\Gamma_0)$. Let $U$ be
  the subset of $\{t_{i+1},\ldots,t_R\}$ consisting of splits that are
  ready for $\Gamma_i'$. Pick $t \in U$ such  that $t$ destroys the vertex
  of $\Gamma_i'$ with the lowest label. Set \[ D_{i+1} = [u_1, \ldots, u_i,
  t, t_{i+1}, \ldots, \hat{t}, \ldots t_R]. \] Set $u_{i+1} = t$ and let
  $\Gamma_{i+1}'=[u_1,\ldots,u_{i+1}](\Gamma_0)$. Let $M_i$ and $N_i$ be
  the maximal vertex and edge label of $\Gamma_i'$. Let $f$ be the edge in
  $\Gamma_{i+1}'$ created by $t$. Label $f$ by $N_i+1$. If $f$ is a loop,
  then label its vertex by $M_i+1$. If $f$ is not a loop, then label
  $\init(f)$ by $M_i+1$ and $\term(f)$ by $M_i+2$. All other vertices and
  edges of $\Gamma_{i+1}'$ will inherit their label from $\Gamma_i'$. This
  completes the construction.

  Since $\Gamma_0$ has $2n-2$ vertices and at most two vertices are created
  in each stage, the maximal vertex label of $\Gamma_R$ is at most
  $2n-2+2R$. Similarly, the maximal edge label of $\Gamma_R$ is at most
  $3n-3+R$. Let
  \[ V = \{1,\ldots, 2n-2+2R \}, \quad E = \{1, \cdots, 3n-3 + R\}. \] Let
  $F_R$ be the set of all pairs of functions $(\phi,\psi)$, where $\phi
  \from V \to \{0,1,2,3\}$ and $\psi \from E \to \{0,1,2,3\}$. 
  
  \begin{proposition}
    
    There exists an embedding of $B_R(\Gamma_0)$ into $F_R$, hence
    $\#B_R(\Gamma_0) \le \#F_R$.

  \end{proposition}
  
  \begin{proof}
    
    By definition, any $\Gamma \in B_R(\Gamma_0)$ can be obtained from
    $\Gamma_0$ by a derivation $D$ of length $r \le R$. Let $D_r$ be the
    canonical derivation of $\Gamma$ coming from $D$. We now \emph{encode}
    $D_r$ by a pair of maps $\phi \from V \to \{0,1,2,3\}$ and $\psi
    \from E \to \{0,1,2,3\}$. Set $D_r$ \[ \Gamma_0 \splitt{u_1} \Gamma_1
    \splitt{u_2} \cdots \splitt{u_r} \Gamma_r. \]

    For $i \in V$, let $j$ be the largest index so that $\Gamma_j$ has a
    vertex $v$ with label $i$. If $j \ge R$ ($j > R$ means no such label
    exists), then set $\phi(i) =0$. If $j < R$, then $u_{j+1}$ must destroy
    $v$, so the support of $u_{j+1}$ is an edge $e$ in $\Gamma_j$ where $v$
    is either $\init(e)$ or $\term(e)$. Define $\phi(i) \in \{1,2,3\}$ to
    be the label of the half-edge of $e$ containing vertex $v$. If $e$ is a
    loop, then choose $\phi(i)$ to be the label of any half-edge of $e$.
    For $i \in E$, let $k$ be the largest index so that $\Gamma_k$ has an
    edge $e$ of label $i$. If $k\ge R$, then $\psi(i)=0$. If $k<R$, then
    $e$ is the support of $t_{k+1}$. In this case, define $\psi(i)\in
    \{0,\ldots,3\}$ to be the label of $t_{k+1}$. 
    
    To see this is an embedding, we will give a \emph{decoding} procedure
    that will recover from $(\phi,\psi)$ the canonical derivation
    $D_r=[u_1,\ldots,u_r]$ and hence $\Gamma$.  
    
    For each $k \ge 0$, suppose $\Gamma_k'$ has been constructed. In
    $\Gamma_k'$, let $i$ range from $1$ to $2n-2+2k$ in order and let $v$
    be the vertex in $\Gamma_k$ with label $i$. If $\phi(i) = 0$, then move
    on to $i+1$. Otherwise, $\phi(i)$ determines a unique edge $e$, where
    $v$ is either the initial or terminal vertex of $e$, such that the
    half-edge of $e$ at $v$ has label $\phi(i)$. We now explain a
    \emph{matching} procedure that can occur in two ways. If $e$ is a loop,
    then we have a match. If $e$ is not a loop, then let $w$ be the other
    vertex of $e$ with label $i'$. If $\phi(i')$ is exactly the label of
    the half-edge of $e$ at $w$, then we have a match. In all other case,
    there is no match and we move on to $i+1$. If there is a match, then
    let $j \in E$ be the label of $e$. The configuration $\psi(j)$
    determines a split supported on $e$ which we call $u_{k+1}'$, and
    applying $u_{k+1}'$ to $\Gamma_k'$ yields $\Gamma_{k+1}'$. Proceed this
    way until $k=r$ results in a derivation $D'$. 
    
    To see that $D'=D_r$. Let $e$ be the support of $u_k$. The encoding
    procedure ensures that the values of $\phi$ on the labels of $\init(e)$
    and $\term(e)$ determine $e$, and hence a match, and the value of
    $\psi$ on the label of $e$ agrees with $u_k$. Furthermore, since only
    the splits that are ready at $\Gamma_{k-1}$ can determine a match, and
    $u_k$ is the unique one among them that destroys the vertex of
    $\Gamma_{k-1}$ with the smallest label, the match coming from $u_k$
    will always be the first match the decoding procedure finds. This shows
    $\Gamma_k = \Gamma_k'$ and $t_k = t_k'$ for all $k$. Therefore,
    $B_R(\Gamma_0)$ embeds in $F_R$. \qedhere 

  \end{proof}
  
  Since $\#F_R = 4^{5n-5+3R}$, $\lim_{R \to \infty} \frac{ \log \#F_R}{R} =
  3 \log 4$. This completes the proof of \thmref{Graph}.

  \subsection{Triangulations of a surface}
  
  Let $\Sigma=\Sigma_{g,p}$ be an oriented surface of genus $g$ with $p$
  punctures. For any $n \ge p$, let $\T = \T(\Sigma)$ be the set of homotopy
  classes of triangulations of $\Sigma$ with $n$ vertices (the punctures of
  $\Sigma$ are always vertices of triangles.) A natural transformation of a
  triangulation is a \emph{diagonal flip}. Given $T \in \T$. Let $\Delta$
  and $\Delta'$ be two triangles in $T$ that share a common edge $E$. View
  $\Delta \cup \Delta'$ as a quadrilateral with diagonal $E$. Replace $E$
  by the other diagonal in the quadrilateral yields a triangulation $T' \in
  \T$. Call this process a (diagonal) flip about $E$ and denote it by $T
  \splitt{d} T'$. Let $B_R(T)$ be the set of all triangulations of $\Sigma$
  obtained from $T$ by a sequence of at most $R$ diagonal flips. 

  Fix $T_0 \in \T$. Dual to $T_0$ is a trivalent graph $\Gamma_0$ obtained
  by putting a vertex in the interior of each triangle and connecting two
  vertices by an edge when two triangles share an edge. Pick a vertex $x_0$
  in $\Gamma_0$ and let $G = \pi_1(\Sigma,x_0)$. We will label each edge of
  $\Gamma_0$ by an element of $G$ as follows. Orient the edges of
  $\Gamma_0$ arbitrarily. Pick a spanning tree $K_0$ in $\Gamma_0$ and
  label each edge of $K_0$ by $1$. Each edge $e$ in the complement of $K_0$
  represent an element of $G$: connect the end points of $e$ to $x_0$ along
  $K_0$ and orient the resulting closed curve so that it matches the
  orientation of $e$ in $\Gamma_0$. Now, label $e$ by the element that this
  closed curve represents in $G$. This makes a $G$--labeled graph $\Gamma_0
  \in \calG_m(G)$, where $m=2g+n-1$ is the rank of $\Gamma_0$. 
  
  By a pair $(\Gamma,f)$ we will mean a $G$--labeled graph $\Gamma \in
  \calG_m(G)$ together with an embedding $ f \from \Gamma \to \Sigma$. We
  say a pair $(\Gamma,f)$ is \emph{well-labeled} if for any closed path $p$
  in $\Gamma$, the product of labels of edges along $p$ is in the conjugacy
  class in $G$ represented by $f(p)$. By construction, $(\Gamma_0,i)$, where
  $\Gamma_0$ is the dual graph to $T_0$ and $i$ is the inclusion map, is
  well-labeled.
  
  \begin{proposition} \label{Prop:T-to-G}
    
    There exists an embedding of $B_R(T_0)$ into $B_R(\Gamma_0)$, hence
    $\#B_R(T_0) \le \#B_R(\Gamma_0)$.

  \end{proposition}
  
  \begin{proof}

    Assume $T$ and a well-labeled dual graph $(\Gamma,i)$ are given.
    Consider a flip $T \splitt{d} T'$ about an edge $E$ in $T$ and let $e$
    be the edge in $\Gamma$ dual to $E$. Identify the quadrilateral
    containing $E$ and a contractible neighborhood of $e$ dual to the
    quadrilateral with the left-hand side of \figref{Flip}. We define a
    split move $\Gamma \splitt{s} \Gamma'$ supported on $e$, where
    $(\Gamma',i)$ is also embedded in $\Sigma$, as indicated by
    \figref{Flip}. We refer to $s$ as the split associated to the flip $d$.
    Note that a closed path $p$ in $\Gamma$ can naturally be mapped to a
    homotopic closed path $p'$ in $\Gamma'$ and the products of labels
    along edges of $p$ and $p'$ are the same. That is, the pair
    $(\Gamma',i)$ is still well-labeled.      
    
    \begin{figure}[htp!]
    \begin{center}
      \includegraphics{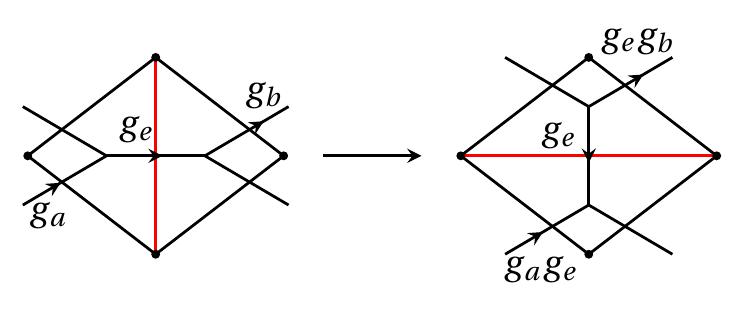}
    \end{center}
    \caption{Flip and dual split.}
    \label{Fig:Flip}
    \end{figure}
    
    We now define a map from $B_R(T_0)$ to $B_R(\Gamma_0)$. For any $T \in
    B_R(T_0)$, choose an arbitrary sequence of flips $T_0 \splitt{d_1} T_1
    \splitt{d_2} \cdots \splitt{d_R} T_R=T$ and let $\Gamma_0 \splitt{s_1}
    \Gamma_1 \splitt{s_2} \cdots \splitt{s_R} \Gamma_R$ be the associated
    sequence of \emph{dual splits} as constructed above. The map from
    $B_R(T_0)$ to $B_R(\Gamma_0)$ is defined by sending $T_R$ to
    $(\Gamma_R,i)$ and then to $\Gamma_R$. 
 
    We show that this map is injective. In fact, for triangulations $T$ and
    $T'$ and dual labeled graphs $(\Gamma,i)$ and $(\Gamma',i)$ that are
    well-labeled, we show that if $\Gamma$ and $\Gamma'$ are equivalent
    $G$--labeled graphs, then there exists a homeomorphism of $\Sigma$
    homotopic to the identity taking $T$ to $T'$.

    Since $\Gamma$ and $\Gamma'$ are equivalent, there is a graph
    isomorphism $\phi \from \Gamma \to \Gamma'$ such that the label of any
    edge $e \in \Gamma$ matches the label of $\phi(e) \in \Gamma'$. Since
    $\Gamma$ and $\Gamma'$ are dual graphs to the triangulations $T$ and
    $T'$ respectively, we can build a homeomorphism $f \from \Sigma \to
    \Sigma$ mapping a triangle of $T$ associated to a vertex $v \in \Gamma$
    to the triangle of $T'$ associated to the vertex $\phi(v)$. To show
    that $\phi$ is homotopic to identity, it is sufficient to show that
    every closed path $q$ in $\Sigma$ is homotopic to $f(q)$.

    First perturb $q$ so it missed the vertices of $T$. Then $q$ can be
    pushed to a closed path $p$ in $\Gamma$. Since $q$ is homotopic to $p$,
    we have $f(q)$ is homotopic to $p' = f(p)$. But the product of labels
    along the closed paths $p$ and $p'$ are identical, which means $p$ and
    $p'$ represent the same conjugacy class in $G$ and hence are homotopic.
    This finishes the proof.  \qedhere
   
  \end{proof}

  \thmref{IntroTriangle} from the introduction now follows from \thmref{Graph} 
  and \propref{T-to-G}. 

%section{Bibliography}
 
  \bibliographystyle{alpha} 
  \bibliography{main}

\begin{thebibliography}{GdlH97}

\bibitem[AP14]{AP14}
Jayadev~S. Athreya and Amritanshu Prasad.
\newblock Growth in right-angled groups and monoids.
\newblock Preprint, 2014.
\newblock Available at {\tt arXiv:1409.4142 [math.GR]}.

\bibitem[Art47]{Art47}
E.~Artin.
\newblock Theory of braids.
\newblock {\em Ann. of Math. (2)}, 48:101--126, 1947.

\bibitem[Bir74]{Bir74}
Joan~S. Birman.
\newblock {\em Braids, links, and mapping class groups}.
\newblock Princeton University Press, Princeton, N.J.; University of Tokyo
  Press, Tokyo, 1974.
\newblock Annals of Mathematics Studies, No. 82.

\bibitem[FM12]{FM12}
Benson Farb and Dan Margalit.
\newblock {\em A primer on mapping class groups}, volume~49 of {\em Princeton
  Mathematical Series}.
\newblock Princeton University Press, Princeton, NJ, 2012.

\bibitem[GdlH97]{GdlH97}
R.~Grigorchuk and P.~de~la Harpe.
\newblock On problems related to growth, entropy, and spectrum in group theory.
\newblock {\em J. Dynam. Control Systems}, 3(1):51--89, 1997.

\bibitem[HM95]{HM95}
Susan Hermiller and John Meier.
\newblock Algorithms and geometry for graph products of groups.
\newblock {\em J. Algebra}, 171(1):230--257, 1995.

\bibitem[Lic64]{Lic64}
W.~B.~R. Lickorish.
\newblock A finite set of generators for the homeotopy group of a
  {$2$}-manifold.
\newblock {\em Proc. Cambridge Philos. Soc.}, 60:769--778, 1964.

\bibitem[LS77]{LS77}
Roger~C. Lyndon and Paul~E. Schupp.
\newblock {\em Combinatorial group theory}.
\newblock Springer-Verlag, Berlin-New York, 1977.
\newblock Ergebnisse der Mathematik und ihrer Grenzgebiete, Band 89.

\bibitem[Mac03]{Mac03}
David J.~C. MacKay.
\newblock {\em Information theory, inference and learning algorithms}.
\newblock Cambridge University Press, New York, 2003.

\bibitem[McM14]{McM14}
Curtis McMullen.
\newblock Entropy and the clique polynomial.
\newblock {\em J. Topol.}, 7(29), 2014.

\bibitem[MKS66]{MKS66}
Wilhelm Magnus, Abraham Karrass, and Donald Solitar.
\newblock {\em Combinatorial group theory: {P}resentations of groups in terms
  of generators and relations}.
\newblock Interscience Publishers [John Wiley \& Sons, Inc.], New
  York-London-Sydney, 1966.

\bibitem[RT13]{RT13}
Kasra Rafi and Jing Tao.
\newblock The diameter of the thick part of moduli space and simultaneous
  {W}hitehead moves.
\newblock {\em Duke Math. J.}, 162(10):1833--1876, 2013.

\bibitem[Sco07]{Sco07}
Richard Scott.
\newblock Growth series for vertex-regular {CAT}(0) cube complexes.
\newblock {\em Algebr. Geom. Topol.}, 7:285--300, 2007.

\bibitem[STT92]{STT92}
Daniel~D. Sleator, Robert~E. Tarjan, and William~P. Thurston.
\newblock Short encodings of evolving structures.
\newblock {\em SIAM J. Discrete Math.}, 5(3):428--450, 1992.

\end{thebibliography}

  \end{document}